\documentclass{article}
\usepackage[utf8]{inputenc}
\usepackage{mathtools}
\usepackage{amsthm}
\usepackage{amsmath, amssymb}
\usepackage{cite}
\usepackage{comment}
\usepackage{authblk}
\usepackage{hyperref}
\usepackage{float}
\usepackage{fontawesome5}
\usepackage[dvipsnames]{xcolor}

\newtheorem{lemma}{Lemma}
\newtheorem{theorem}{Theorem}
\newtheorem{corollary}{Corollary}

\begin{document}
\title{A counterexample to the conjecture on Biclique Partition number of Split Graphs and related problems}

\author{Anand Babu \quad
        Ashwin Jacob}
\affil{Department of Computer Science {\&} Engineering, \\ National Institute of Technology Calicut, Kozhikode, India.}
\date{}
\maketitle

\begin{abstract}
The biclique partition number of a graph \(G\), denoted \( \operatorname{bp}(G)\), is the minimum number of biclique subgraphs needed to partition the edge set of $G$. Lyu and Hicks \cite{lyu2023finding} posed the open problem of whether \(  \operatorname{bp}(G)  = \operatorname{mc}(G^c) - 1 \) holds for every co-chordal graph or split graph, where \( \operatorname{mc}(G^c) \) denotes the number of maximal cliques in the complement of \( G \). Such a result would extend the celebrated Graham--Pollak theorem to a more general class of graphs. In this note, we answer this problem in the negative by providing a counterexample using a split graph. We also construct an infinite family of counterexamples and prove some structural properties of biclique partitions of split graphs. Finally, we solve an open problem posed by Siewert \cite{siewert2000biclique} on the existence of singular \(n\)-tournaments with binary rank \(n\).
\end{abstract}

\section{Introduction}
The \emph{biclique partition number} of a graph \(G\), denoted \(\operatorname{bp}(G)\), is the minimum number of bicliques whose edge sets partition \(E(G)\). Determining \(\operatorname{bp}(G)\) is NP-hard for general graphs~\cite{kratzke1988eigensharp}. For the complete graph \(K_n\), partitioning the edges into stars centered at \(n-1\) vertices yields the upper bound \(\operatorname{bp}(K_n)\le n-1\). The classical result of Graham and Pollak~\cite{graham1971addressing,graham1972embedding,graham1978distance} shows that any partition of \(E(K_n)\) into bicliques requires at least \(n-1\) subgraphs. Their proof uses Sylvester's law of inertia~\cite{graham1971addressing,graham1972embedding}. Other proofs were provided by Tverberg~\cite{tverberg1982decomposition} and Peck~\cite{peck1984new} using linear-algebraic techniques. Later proofs using the polynomial space method and a counting argument were given by Vishwanathan~\cite{vishwanathan2008polynomial,vishwanathan2013counting}.

A natural generalization is to determine the minimum number of complete \(r\)-partite \(r\)-uniform hypergraphs needed to partition the edge set of the complete \(r\)-uniform hypergraph \(K_n^{(r)}\). We denote this minimum by \(\operatorname{bp}_r(n)\). For \(r=2\), this reduces to \(\operatorname{bp}(K_n)=\operatorname{bp}_2(n)\). For \(r>2\), this problem was posed by Aharoni and Linial~\cite{alon1986decomposition}. Alon ~\cite{alon1986decomposition} later proved that \(\operatorname{bp}_3(n)=n-2\) and, for fixed \(r\ge 4\), established the existence of positive constants \(c_1(r)\) and \(c_2(r)\) such that
\(
c_1(r) \cdot n^{\lfloor r/2 \rfloor} \leq \operatorname{bp}_r(n) \leq c_2(r) \cdot n^{\lfloor r/2 \rfloor}.
\)
Subsequent improvements in the lower-order terms were obtained by Cioab\u{a}, Küngden, and Verstraëte~\cite{cioabua2009decompositions}. More recently, Leader, Milićević, and Tan~\cite{leader2017decomposing} provided asymptotic refinements for the upper bound constant \( c_2(r) \), with further advances appearing in~\cite{leader2018improved,babu2019bounds}.

Further extensions of the Graham--Pollak theorem have been studied. Let \(L=\{l_1,\dots,l_p\}\) with \(l_i\in\mathbb{Z}^+\). An \emph{\(L\)-bipartite covering} of a graph \(G\) is a family of bicliques such that every edge of \(G\) appears in exactly \(l_i\) bicliques for some \(l_i\in L\). The minimum size of such a family is the \emph{\(L\)-bipartite covering number} of \(G\), denoted \(\operatorname{bp}_L(G)\). For \(K_n\) and \(L=\{1\}\), we have \(\operatorname{bp}_{\{1\}}(K_n)=\operatorname{bp}(K_n)\). The case in which every edge is covered exactly \(\lambda\) times (i.e., \(\operatorname{bp}_2(n,\{\lambda\})\)) was studied by De Caen, Gregory, and Pritikin~\cite{de1993minimum}. For lists \(L\) consisting of all odd integers, \(\operatorname{bp}_2(n,L)\) was investigated by Radhakrishnan, Sen, and Vishwanathan~\cite{radhakrishnan2000depth}. Asymptotically tight bounds for several lists appear in~\cite{cioabua2013variations,babu2022improved}. Alon \textit{et al.}~\cite{alon2023new} improved bounds for \(L=\{1,2\}\); further improvements appear in~\cite{grytczuk2024neighborly}. Exact bounds for some odd lists appear in Buchanan \textit{et al.}~\cite{buchanan2023odd,buchanan2024odd}. Leader and Tan also provided improvements for lists of odd integers~\cite{leader2024odd}. Generalizations for \(L=\{1,\dots,p\}\) were explored by Cioab\u{a} and Tait~\cite{cioabua2013variations}. Extensions of this to hypergraphs can be found in ~\cite{babu2021multicovering,babu2022improved}. For the list \(L=\{1,\cdots,p\}\), an \emph{\( r \)-partite \( p \)-multicover} of the complete \( r \)-uniform hypergraph \( K_n^{(r)} \) is a collection of complete \( r \)-partite \( r \)-graphs such that every hyperedge is covered exactly \( \ell \) times for some \( \ell \in L \). The \emph{ \(r\)--partite \(p\)-multicovering number} is the minimum size of such a cover. This is denoted by \( \operatorname{bp}_r(n,p) \). The bipartite \( p \)-multicovering problem (i.e., the case \( r = 2 \)) was first introduced by Alon~\cite{alon1997neighborly}.  Note that $\operatorname{bp}_2(n,\{1\}) = \operatorname{bp}_2(n)$. For fixed \( p \geq 2 \), Alon~\cite{alon1997neighborly} established the bounds
\(
(1+o(1))\left(\frac{p!}{2^p}\right)^{1/p} n^{1/p} \leq \operatorname{bp}_2(n, p) \leq (1+o(1))p n^{1/p}.
\)
Huang and Sudakov~\cite{huang2012counterexample} improved the lower bound to
\(
(1+o(1))\left(\frac{p!}{2^{p-1}}\right)^{1/p} n^{1/p}.
\)
The exact values of biclique covering and partitioning numbers are known for only a few graph classes. Precise bounds for Johnson graphs, odd cycles, complete multipartite graphs, and random graphs can be found in~\cite{alon2021addressing}. Related investigations into optimal addressings for other graph families appear in~\cite{elzinga2004addressing, alon2021addressing}.
\paragraph{}
The remainder of the paper is organized as follows. In Section~\ref{prelim}, we present preliminaries and outline the connection between biclique partitions and addressings into the squashed cube. In Section~\ref{main}, we present our main results. We answer the question in negative posed by Lyu and Hicks~\cite{lyu2023finding} by giving counterexamples that are split-graphs, including an infinite family. We then determine the biclique partition number for unbalanced split graphs and prove structural properties of biclique partitions in balanced split graphs. Finally, we prove the existence of singular \(n\)-tournaments with binary rank \(n\), answering an open problem posed by Siewert~\cite{siewert2000biclique}.

\section{Preliminaries}\label{prelim}
Let \( G \) be a graph. We denote its vertex set by \( V(G) \) and its edge set by \( E(G) \). For any subset \( S \subseteq V(G) \), the \emph{subgraph induced by \( S \)} is denoted by \( G[S] \).

A \emph{clique} (or complete graph) is a graph where every pair of distinct vertices is adjacent. A clique on \( n \) vertices is denoted by \( K_n \). A \emph{maximal clique} is a clique that is not a proper subset of any larger clique in the graph. A \emph{maximum clique} is a clique of the largest possible size in the graph; its size is referred to as the \emph{clique number}, denoted \( \omega(G) \).

A \emph{biclique} (or complete bipartite graph) is a graph whose vertex set can be partitioned into two disjoint subsets, \( V_1 \) and \( V_2 \), such that every vertex in \( V_1 \) is adjacent to every vertex in \( V_2 \), and no edges exist between vertices within the same subset. A biclique with \( |V_1| = m \) and \( |V_2| = n \) is denoted by \( K_{m,n} \).

A \emph{star} is a biclique with one part of size $1$ (the \emph{center}), and the other part contains all remaining vertices.

An \emph{independent set} in a graph \( G \) is a set of vertices no two of which are adjacent. A \emph{maximal independent set} is one that cannot be extended by including any additional vertex from \( G \) without losing the independence property. A \emph{maximum independent set} is an independent set of the largest possible size in \( G \), and its cardinality is called the \emph{independence number}, denoted \( \alpha(G) \).

A \emph{tournament} is an orientation of a complete graph. Equivalently, a tournament on vertex set \([n]\) is a directed graph in which, for every pair of distinct vertices \(i\) and \(j\), exactly one of the arcs \(i\to j\) or \(j\to i\) is present. The adjacency matrix \(A\) of a tournament is the \(\{0,1\}\)-matrix with \(A_{ij}=1\) if and only if \(i\to j\).

A tournament is \emph{regular} if all its vertices have the same score (outdegree). Equivalently, a regular \(n\)-tournament matrix is a tournament matrix with equal row sums. If \(A\) is a regular \(n\)-tournament, then \(n\) must be odd. A tournament is \emph{near-regular} if the maximum difference between two scores is \(1\). Thus, a near-regular \(n\)-tournament matrix is a tournament matrix in which half of the row sums equal \(n/2-1\) and half equal \(n/2\). If \(A\) is a near-regular \(n\)-tournament, then \(n\) must be even.

For a graph \(G\), we write \(G^c\) for its complement and \(N_G(v)\) for the (open) neighborhood of a vertex \(v\).

A graph \( G \) is called a \emph{split graph} if its vertex set admits a partition \( V(G) = K \cup S \), where \( G[K] \) is a clique, and \( G[S] \) is an independent set. Such a partition is referred to as a \emph{split partition}. A split graph is said to be \emph{balanced} if there exists a split partition in which \( |K| = \omega(G) \) and \( |S| = \alpha(G) \); otherwise, it is called \emph{unbalanced}. A split partition is said to be \emph{\( S \)-max} if \( |S| = \alpha(G) \) and \emph{\( K \)-max} if \( |K| = \omega(G) \). In balanced split graphs, the split partition that satisfies both \( |K| = \omega(G) \) and \( |S| = \alpha(G) \) is unique.

The following theorem is a consequence of the results of Hammer and Simeone~\cite{hammer1981splittance} and is presented in~\cite{golumbic2004algorithmic, collins2018finding}.

\begin{theorem}[Hammer and Simeone\cite{hammer1981splittance}]\label{theorem1}
Let \( G \) be a split graph with a vertex partition \( V(G) = K \cup S \). Then, exactly one of the following holds.
\begin{enumerate}
    \item \( |K| = \omega(G) \) and \( |S| = \alpha(G) \). \hfill (balanced)
    \item \( |K| = \omega(G) - 1 \) and \( |S| = \alpha(G) \). \hfill (unbalanced, \( S \)-max)
    \item \( |K| = \omega(G) \) and \( |S| = \alpha(G) - 1 \). \hfill (unbalanced, \( K \)-max)
\end{enumerate}
Moreover, in case~(2), there exists a vertex \( s \in S \) such that \( K \cup \{s\} \) forms a clique, and in case~(3), there exists a vertex \( k \in K \) such that \( S \cup \{k\} \) is an independent set.
\end{theorem}

\paragraph{}
For a \{0, 1\} matrix \(M\) of dimensions \(n \times m\), consider the following three definitions of rank.
The (standard) rank of \(M\) over \( \mathbb{R}\), denoted by \(rank_R(M)\), is the minimal \(k\) for which there exist real matrices \(A\) and \(B\) of dimensions \(n \times k\) and \(k \times m\) respectively, such that \(M = A \cdot B\) where the operations are over \( \mathbb{R}\). The binary rank of \(M\), denoted by \( rank_{bin}(M)\), is the minimal \(k\) for which there exist
\{0, 1\} matrices \(A\) and \(B\) of dimensions \(n \times k\) and \(k \times m\) respectively, such that \(M = A \cdot B\) where the operations are over \( \mathbb{R}\). Equivalently, \(rank_{bin}(M)\) is the smallest number of monochromatic combinatorial rectangles in a partition of the ones in M. The \emph{non-negative integer rank} of \(M\), denoted \(r_{\mathbb{Z}^+}(M)\), is the smallest integer \(k\) for which there is an \(n \times k\) matrix \(B\) and \(k\times m\) matrix \(C\) over the positive integers such that \(M = A \cdot B\). If $M$ is a $\{0,1\}$ matrix then $A$ and $B$ are $\{0,1\}$ matrices. Consequently, the binary rank and nonnegative integer rank of the $\{0,1\}$ matrices are equal. The term rank of \( M\), denoted by \( rank_t(M)\), is the smallest number of rows and columns containing all of the nonzero entries of \(M\).

The following is the celebrated theorem of Graham and Pollak~\cite{graham1972embedding}.
\begin{theorem}[Graham-Pollak\cite{graham1972embedding}]\label{grahampollak}
    The minimum number of bicliques required to partition the edge set of a complete graph on $n$ vertices is $n-1$.
\end{theorem}

\subsection*{Biclique covers and Addressing}\label{def2}
Let \( \Sigma = \{0, 1, \ast\} \) be an alphabet, where the symbol \( \ast \) is called a \emph{joker}. For \( d \in \mathbb{N}\), let \( \Sigma^d \) denote the set of all strings of length \( d \) over \( \Sigma \). Each string \( x \in \Sigma^d \) defines a subcube \( H(x) \subseteq \{0,1\}^d \) consisting of all binary strings obtained by replacing each joker in \( x \) with \( 0 \) or \( 1 \) in all possible ways. We call the number of jokers in \( x \), denoted \( j(x) \), the \emph{dimension} of the subcube \( H(x) \). Hence \( |H(x)| = 2^{j(x)} \).

For two strings \( x, y \in \Sigma^d \), we define the distance \( d(x,y) \) as the number of coordinates where one string has \( 0 \) and the other has \( 1 \). Note that if \( d(x,y) \geq 1 \), then the subcubes \( H(x) \) and \( H(y) \) are disjoint. Moreover, if the dimensions of \( H(x) \) and \( H(y) \) are \( i \) and \( j \) respectively, then any two binary strings \( u \in H(x) \) and \( v \in H(y) \) differ in at most \( d(x,y) + i + j \) coordinates.

An \emph{addressing} of a graph \( G = (V, E) \) into a squashed \(m\)-cube is a mapping \( f \colon V \to \Sigma^m \) that is distance-preserving. Hence, \(d(f(u), f(v)) \geq 1\) for all edges \(uv \in E\). Graham and Pollak showed that every connected graph admits such an addressing. The minimum value of \( m \) for which such an addressing exists is called the \emph{squashed cube dimension} of \( G \). For the complete graph \( K_n \), Graham and Pollak proved that the squashed cube dimension equals \( n - 1 \), and this value equals the minimum number of bicliques required to partition the edge set of \( K_n \).

We interpret strings in \( \Sigma^d \) through their connection to axis-aligned boxes. A family of axis-aligned boxes in \(\mathbb{R}^d\) is \(k\)-neighborly if the intersection of every two boxes has dimension at least \(d-k\) and at most \(d-1\). Let \(n(k, d)\) denote the maximum size of such a family. As explained in \cite{alon1997neighborly}, \(n(k, d)\) equals the maximum number of vertices in a complete graph whose edge set can be covered by \(d\) bicliques, with each edge covered at least once and at most \(k\) times. Equivalently, \(n(k, d)\) is the maximum size of a family \(F \subseteq \Sigma^d\) such that \(1 \leq d(x, y) \leq k\) for every pair of distinct strings \(x, y \in F\). For $k=1$, \(n(k,d)\) is precisely obtained by using the Graham-Pollak theorem.

Given such a family \( F \), we define its \emph{volume} as
\[
\operatorname{vol}(F) = \sum_{x \in F} 2^{j(x)},
\]
which equals the total number of binary strings covered by \(\bigcup_{x \in F} H(x)\). Geometrically, \(2^{j(x)}\) represents the volume of the standard box corresponding to \( x \), and \(\operatorname{vol}(F)\) is the total volume of all boxes in the family.

The connection between biclique coverings and addressings is established as follows. Given a biclique covering \(\mathcal{B} = \{B_1,\dots,B_{d}\}\) of \(E(K_n)\), we obtain an addressing \( f \colon V(K_n) \to \Sigma^d \) by assigning to each vertex \(v\) a string \(f(v) \in \Sigma^{d}\) where the \(i\)-th coordinate is \(0\) if \(v\) belongs to the first part of \(B_i\), is \(1\) if \(v\) belongs to the second part of \(B_i\), and is \(\ast\) if \(v\) does not belong to \(B_i\). Therefore, \(n(k, d)\) is the maximum possible number of strings in a family $F \subseteq \Sigma^{d}$ so that $1 \leq d(x, y) \leq k$ holds for any two distinct $x, y \in F$.

\section{Main Result}\label{main}
\begin{lemma}\label{counterexample}
There exists a split graph $G$ with vertex partition $V(G)=K \cup S$, where $K$ is a clique and $S$ is an independent set, such that \(\operatorname{bp}(G)=\operatorname{mc}(G^c)-2\).
\end{lemma}

\begin{proof}
Let $G$ be the split graph with adjacency matrix $A$, where the vertices $\{1,\ldots,7\}$ form the clique $K$ and the vertices $\{8,\ldots,14\}$ form the independent set $S$.

    \begin{table}[H]
    \begin{center}
\begin{tabular}{|l|l|l|l|l|l|l|l|l|l|l|l|l|l|l|}
\hline
   & 1 & 2 & 3 & 4 & 5 & 6 & 7 & 8 & 9 & 10 & 11 & 12 & 13 & 14 \\ \hline
1  & 0 & 1 & 1 & 1 & 1 & 1 & 1 & 0 & 1 & 0  & 0  & 0  & 0  & 0  \\ \hline
2  & 1 & 0 & 1 & 1 & 1 & 1 & 1 & 0 & 0 & 1  & 1  & 0  & 0  & 0  \\ \hline
3  & 1 & 1 & 0 & 1 & 1 & 1 & 1 & 1 & 0 & 0  & 0  & 1  & 0  & 0  \\ \hline
4  & 1 & 1 & 1 & 0 & 1 & 1 & 1 & 1 & 0 & 1  & 0  & 0  & 1  & 0  \\ \hline
5  & 1 & 1 & 1 & 1 & 0 & 1 & 1 & 1 & 1 & 0  & 1  & 0  & 1  & 0  \\ \hline
6  & 1 & 1 & 1 & 1 & 1 & 0 & 1 & 1 & 1 & 1  & 0  & 0  & 0  & 1  \\ \hline
7  & 1 & 1 & 1 & 1 & 1 & 1 & 0 & 1 & 1 & 1  & 1  & 1  & 0  & 0  \\ \hline
8  & 0 & 0 & 1 & 1 & 1 & 1 & 1 & 0 & 0 & 0  & 0  & 0  & 0  & 0  \\ \hline
9  & 1 & 0 & 0 & 0 & 1 & 1 & 1 & 0 & 0 & 0  & 0  & 0  & 0  & 0  \\ \hline
10 & 0 & 1 & 0 & 1 & 0 & 1 & 1 & 0 & 0 & 0  & 0  & 0  & 0  & 0  \\ \hline
11 & 0 & 1 & 0 & 0 & 1 & 0 & 1 & 0 & 0 & 0  & 0  & 0  & 0  & 0  \\ \hline
12 & 0 & 0 & 1 & 1 & 0 & 0 & 1 & 0 & 0 & 0  & 0  & 0  & 0  & 0  \\ \hline
13 & 0 & 0 & 0 & 0 & 1 & 0 & 0 & 0 & 0 & 0  & 0  & 0  & 0  & 0  \\ \hline
14 & 0 & 0 & 0 & 0 & 0 & 1 & 0 & 0 & 0 & 0  & 0  & 0  & 0  & 0  \\ \hline
\end{tabular}
\end{center}
\end{table}
\paragraph{}
The graph $G$ contains an induced clique on $7$ vertices. Since $\operatorname{bp}(G)\ge \operatorname{bp}(G')$ for every induced subgraph $G'$ of $G$, the Graham--Pollak theorem implies that $\operatorname{bp}(G)\ge 6$.
Now consider the following collection of bicliques, $\mathcal{B}=\{B_1,\dots,B_6\}$. This forms a biclique partition of $G$.
\begin{align*}
B_1(U_1,V_1) &= \{4,5\}     \cup \{1,6,8,13\} \\
B_2(U_2,V_2) &= \{6\}       \cup \{1,2,3,7,8,9,10,14\} \\
B_3(U_3,V_3) &= \{3,7\}     \cup \{1,5,8,12\} \\
B_4(U_4,V_4) &= \{1,5,7\}   \cup \{2,9\} \\
B_5(U_5,V_5) &= \{2,4,7\}   \cup \{3,10\} \\
B_6(U_6,V_6) &= \{2,5,7\}   \cup \{4,11\}
\end{align*}

Therefore, $\operatorname{bp}(G)=6$. We claim that $\operatorname{mc}(G^c)=8$.
To prove this, note that $K$ is a clique and $S$ is an independent set in $G$. Hence, $S$ is a clique and $K$ is an independent set in $G^c$. Thus, every clique of $G^c$ contains at most one vertex of $K$.
Each vertex $k\in K$ has at least one neighbor in $S$ in $G$. Therefore, $k$ is nonadjacent in $G^c$ to at least one vertex of $S$. No vertex of $K$ can be added to $S$ in $G^c$, so $S$ is a maximal clique of $G^c$.
For each $k\in K$, the set $\{k\}\cup (N_{G^c}(k)\cap S)$ is a clique of $G^c$. It is maximal because no second vertex of $K$ can be added.
Therefore, $G^c$ has the maximal clique $S$. It also has, for each $k\in K$, the maximal clique $\{k\}\cup (N_{G^c}(k)\cap S)$. These are the only maximal cliques. Hence, $\operatorname{mc}(G^c)=8$.
\end{proof}
\paragraph{}
We construct the split graphs in Theorem~\ref{infinitefamily} from carefully chosen tournament matrices. We begin with adjacency matrices of regular or near-regular tournaments and then modify these matrices to obtain the required block structure. This construction limits the rank of the tournament blocks and yields the desired biclique partition. Siewert has studied different types of ranks of the tournament-matrix in Section~4.4 of~\cite{siewert2000biclique}.
\begin{theorem}\label{infinitefamily}
There exists a family of split graphs $G$ such that \(\operatorname{bp}(G) = \operatorname{mc}(G^c) - 2\).
\end{theorem}
\begin{proof}
Consider the block adjacency matrix \(P\) of a balanced split graph \(G=K \cup S\) on \(2n\) vertices. Its rows and columns are indexed by the vertices in \(K\), followed by the vertices in \(S\). The vertices in \(K\) are labeled \(\{u_1,\cdots,u_n\}\), and the vertices in \(S\) are labeled \(\{v_1,\cdots,v_n\}\).
\[
P=
\begin{bmatrix}
    J_n-I_n &A_n\\
    A_n^T &0
\end{bmatrix}
\]
where \(J_n\) denotes the \(n\times n\) all-ones matrix, \(I_n\) denotes the \(n\times n\) identity matrix, and \(A_n\) is the adjacency matrix of a tournament, defined below. Note that \(A_n+A_n^T=J_n-I_n\).
\begingroup
\setlength{\arraycolsep}{3pt}
\[
A_n =
\left[
\begin{array}{ccccccccc}
0 & 1 & 0 & 0 & \cdots & \cdots & 0 & 0 & 0 \\
0 & 0 & 0 & 0 & \cdots & \cdots & 0 & 0 & 1 \\
1 & 1 & 0 & 0 & \cdots & \cdots & 0 & 0 & 1 \\
1 & 1 & 1 &        &        &        &        &        & 1 \\
\vdots & \vdots & \vdots &  & \multicolumn{3}{c}{U_m} &  & \vdots \\
\vdots & \vdots & \vdots &  &        &        &        &  & \vdots \\
1 & 1 & 1 &        &        &        &        &        & 1 \\
1 & 1 & 1 &        &        &        &        &        & 1 \\
1 & 0 & 0 & 0 & \cdots & \cdots & 0 & 0 & 0
\end{array}
\right]
\]
\endgroup

\[
\begin{array}{cc}
U_m=
\left[
\begin{array}{ccccccccc}
0 & 0 & 1 & \cdots & \cdots & 1 & 0 & 1 \\
1 & 0 & 0 & \cdots & \cdots & 0 & 1 & 0 \\
0 & 1 & 0 & \cdots & \cdots & 0 & 0 & 1 \\
\vdots & \vdots & \vdots &        &        & \vdots & \vdots & \vdots \\
\vdots & \vdots & \vdots &        &        & \vdots & \vdots & \vdots \\
0 & 1 & 0 & \cdots & \cdots & 0 & 0 & 1 \\
1 & 0 & 1 & \cdots & \cdots & 1 & 0 & 0 \\
0 & 1 & 0 & \cdots & \cdots & 0 & 1 & 0
\end{array}
\right]
&
U_m=
\left[
\begin{array}{ccccccccc}
0 & 0 & 1 & \cdots & \cdots & 0 & 1 & 0 \\
1 & 0 & 0 & \cdots & \cdots & 1 & 0 & 1 \\
0 & 1 & 0 & \cdots & \cdots & 0 & 1 & 0 \\
\vdots & \vdots & \vdots &        &        & \vdots & \vdots & \vdots \\
\vdots & \vdots & \vdots &        &        & \vdots & \vdots & \vdots \\
1 & 0 & 1 & \cdots & \cdots & 0 & 0 & 1 \\
0 & 1 & 0 & \cdots & \cdots & 1 & 0 & 0 \\
1 & 0 & 1 & \cdots & \cdots & 0 & 1 & 0
\end{array}
\right]
\\[1ex]
\\
\text{for $m$ odd} & \text{for $m$ even}
\end{array}
\]

Here \(U_m\) is the adjacency matrix of a tournament on \(m\) vertices, where \(U_m[i,j]=1\) if and only if \(j-i\) is even and positive or odd and negative, and \(U_m[i,j]=0\) otherwise. The matrix \(U_m\) is regular for odd \(m\) and near-regular for even \(m\). Note that \(A_n\) is a tournament matrix in which every row and every column has at least one \(1\). As noted in~\cite{siewert2000biclique},
\[
\operatorname{rank}_{\mathbb{R}}(A_n)=\operatorname{rank}_{\mathrm{bin}}(A_n)=\operatorname{rank}_t(A_n)=m+3=n-1.
\]

Let \(G[K,S]\) be the bipartite graph induced by the edges between \(K\) and \(S\) of the split graph $G$. Its biadjacency matrix is \(A_n\). Since \(\operatorname{rank}_{\mathrm{bin}}(A_n)=n-1\), the $1$-entries of \(A_n\) admit a partition into \(n-1\) monochromatic rectangles. Equivalently, \(E(G[K,S])\) admits a biclique partition \(\mathcal{B}=\{B_1,\dots,B_{n-1}\}\). Write \(B_i\) as \(B_i(U_i,V_i)\) with \(U_i\subseteq K\) and \(V_i\subseteq S\).

We extend \(\mathcal{B}\) to a biclique partition \(\mathcal{B}'\) of \(E(G)\) as follows. For each \(j\in[n]\), add \(u_j\) to \(V_i\) if and only if \(v_j\in V_i\). Since \(K\) is a clique, this modification preserves the biclique property. It covers all edges between \(K\) and \(S\) exactly once, because \(\mathcal{B}\) already partitions \(E(G[K,S])\). It also covers each edge inside \(K\) exactly once, because an edge \(u_a u_b\in E(G[K])\) is covered in the unique biclique \(B_i'\) for which \(u_a\) and \(v_b\) lie in opposite parts of \(B_i\). Therefore, \(\mathcal{B}'\) is a biclique partition of \(E(G)\) of size \(n-1\).

Since every vertex of \(K\) has a neighbor in \(S\) in $G$, each \(u\in K\) is nonadjacent in \(G^c\) to at least one vertex of \(S\). Thus, \(S\) is a maximal clique in \(G^c\). Also, \(K\) is an independent set in \(G^c\), so the remaining maximal cliques are $\{u_j\}\cup (N_{G^c}(u_j)\cap S)$ for \(u_j\in K\). Therefore, \(\operatorname{mc}(G^c)=n+1\), and hence \(\operatorname{bp}(G)=n-1=\operatorname{mc}(G^c)-2\).
\end{proof}

\paragraph{}
In what follows, Lemma~\ref{unbalanced_lemma} determines the biclique partition number for unbalanced split graphs, and in particular shows that the Lyu--Hicks equality \(\operatorname{bp}(G)=\operatorname{mc}(G^c)-1\) holds for this subclass of split graphs. 
\begin{lemma}\label{unbalanced_lemma}
The biclique partition number $\operatorname{bp}(G)$ of an unbalanced split graph $G$ is $\omega(G)-1$.
\end{lemma}
\begin{proof}
We establish equality by proving both bounds. First, we consider a split partition $V(G) = K \cup S$ with $|K| = \omega(G) - 1$ and $|S| = \alpha(G)$. Since $G$ contains an induced clique of size $\omega(G)$ and the biclique partition number of any graph is at least that of its induced subgraphs, we have $\operatorname{bp}(G) \geq \operatorname{bp}(K_{\omega(G)})$. By Theorem \ref{grahampollak}, $\operatorname{bp}(K_{\omega(G)}) = \omega(G) - 1$; thus, $\operatorname{bp}(G) \geq \omega(G) - 1$.

For the upper bound, we construct an explicit partition. For each vertex $v \in K$, let $B_v$ be the star centered at $v$ covering all edges incident to $v$ in the current graph (i.e., the graph after removing edges covered by previous stars). This yields $\omega(G) - 1$ bicliques, because $|K| = \omega(G) - 1$. Each edge $uv$ is covered exactly once when processing the first endpoint in $K$ encountered in the ordering. Edges within $K$ are covered when the first endpoint is processed, whereas the edges between $K$ and $S$ are covered when the vertex in $K$ is processed. Thus $\operatorname{bp}(G) \leq \omega(G) - 1$.
Combining the bounds gives $\operatorname{bp}(G) = \omega(G) - 1$.
\end{proof}

Moreover, for an unbalanced split graph we have \(\omega(G)-1=\operatorname{mc}(G^c)-1\): indeed, in the complement \(G^c\), the set \(S\) induces a clique and \(K\) induces an independent set, so every maximal clique of \(G^c\) contains at most one vertex of \(K\). Since (in the \(S\)-max split partition) \(|K|=\omega(G)-1\), the maximal cliques of \(G^c\) are \(S\) together with \(|K|\) cliques of the form \(\{k\}\cup (N_{G^c}(k)\cap S)\) for \(k\in K\). Hence \(\operatorname{mc}(G^c)=|K|+1=\omega(G)\), and therefore \(\operatorname{mc}(G^c)-1=\omega(G)-1\).
\paragraph{}
We now focus on balanced split graphs. In the next lemmas (Lemmas~\ref{structure1_lemma}, \ref{structure2_lemma}, \ref{addressing_lemma_1}, and \ref{volume_lemma_1}), we derive structural restrictions on biclique partitions of size at most \(\omega(G)-1\) and translate these restrictions into properties of the induced addressings.

\begin{lemma}\label{structure1_lemma}
Let $G$ be a balanced split graph with vertex partition $V(G) = K \cup S$, where $K$ is a clique, and $S$ is an independent set. If $\mathcal{B} = \{B_1, \dots, B_r\}$ is a biclique partition of $E(G)$ with $r \leq \omega(G) - 1$, then no $B_i \in \mathcal{B}$ is a star centered at a vertex in $S$.
\end{lemma}

\begin{proof}
Suppose to the contrary that some $B_k \in \mathcal{B}$ is a star centered at $v \in S$. Let $E_k$ be the edge set of $B_k$. Consider the subgraph $G' = (V(G), E(G) \setminus E_k)$. Since $G$ contains a clique of size $\omega(G)$ and removing edges incident to $v \in S$ does not remove edges within $K$ (as $S$ is independent), $G'$ still contains a clique of size $\omega(G)$. 

By Theorem \ref{grahampollak}, any biclique partition of $K_{\omega(G)}$ requires at least $\omega(G) - 1$ bicliques. Since $K_{\omega(G)}$ is an induced subgraph of $G'$, we have $\operatorname{bp}(G') \geq \operatorname{bp}(K_{\omega(G)}) = \omega(G) - 1$. However, $\mathcal{B} \setminus \{B_k\}$ partitions $E(G')$ with $r-1 \leq \omega(G)-2$ bicliques, which contradicts $\operatorname{bp}(G') \geq \omega(G) - 1$.
\end{proof}

\begin{lemma}\label{structure2_lemma}
Let $G$ be a balanced split graph with vertex partition $V(G) = K \cup S$, where $K$ is a clique and $S$ is an independent set. If $\mathcal{B} = \{B_1, \dots, B_r\}$ is a biclique partition of $E(G)$ with $r \leq \omega(G) - 1$, then no $B_i \in \mathcal{B}$ has a part entirely contained in $S$.
\end{lemma}

\begin{proof}
Suppose, to the contrary, that some $B_k \in \mathcal{B}$ has a part $A \subseteq S$. Since $S$ is independent, every edge of $B_k$ has one endpoint in $A$ and the other in $K$. Let $E_k$ be the edge set of $B_k$, and consider the graph $G'=(V(G), E(G)\setminus E_k)$.

The removal of $E_k$ only affects the edges between $S$ and $K$, leaving the clique $K$ completely intact. Therefore, $G'$ contains an induced clique of size $\omega(G)$. By Theorem \ref{grahampollak}, $\operatorname{bp}(K_{\omega(G)}) = \omega(G) - 1$, and thus, $\operatorname{bp}(G') \geq \omega(G) - 1$. However, $\mathcal{B} \setminus \{B_k\}$ partitions $E(G')$ with $r-1 \leq \omega(G) - 2$ bicliques, which contradicts $\operatorname{bp}(G') \geq \omega(G) - 1$.
\end{proof}

\begin{lemma}\label{addressing_lemma_1}
Let $G$ be a balanced split graph with split partition $V(G) = K \cup S$ and let $\mathcal{B} = \{B_1, \dots, B_r\}$ be a biclique partition of $E(G)$ with $r \leq \omega(G)-1$. For the addressing $f \colon V(G) \to \Sigma^r$ induced by $\mathcal{B}$ (where $\Sigma = \{0,1,*\}$), every vertex $u \in K$ has at least one coordinate of $f(u)$ equal to $0$.
\end{lemma}

\begin{proof}
Since $S$ is an independent set, no biclique $B_i \in \mathcal{B}$ contains two vertices from $S$ in different parts, as this implies an edge within $S$. Therefore, without loss of generality, all vertices of $S$ appearing in any $B_i$ are assumed to be in the second part. Consequently, for any $v \in S$, we have $f(v) \in \{1,*\}^r$.

We now consider $u \in K$. Since $G$ is balanced, $u$ must have at least one neighbor $v \in S$; otherwise, $\{u\} \cup S$ forms an independent set of size $\alpha(G)+1$, contradicting $|S| = \alpha(G)$ and $u$ has at least one non-neighbor $w \in S$; otherwise $u$ would be adjacent to all $S$, making $K \cup \{u\}$ a larger clique, contradicting $|K| = \omega(G)$. The edge $uv$ is covered by some biclique $B_j$. Because $v$ is assigned to the second part of $B_j$, $u$ must be in the first part to cover $uv$. Thus, the $j$th coordinate of $f(u)$ is $0$.
\end{proof}

\begin{lemma}\label{volume_lemma_1}
Let $G$ be a balanced split graph with a split partition $V(G) = K \cup S$. Suppose $E(G)$ admits a biclique partition $\mathcal{B} = \{B_1, \dots, B_r\}$ with $r \leq \omega(G)-1$. For the addressing $f \colon V(G) \to \Sigma^r$ (where $\Sigma = \{0,1,*\}$) induced by $\mathcal{B}$, let $F_K = \{f(u) \mid u \in K\}$ be the strings assigned to the clique vertices. Then, $F_K$ is $1$-neighborly, and $\operatorname{vol}(F_K) < 2^r$.
\end{lemma}

\begin{proof}
Consider the induced subgraph, $G[K] \cong K_{\omega(G)}$. The biclique partition $\mathcal{B}$ restricted to the vertices of $K$ partitions $E(G[K])$. By Lemmas \ref{structure1_lemma} and \ref{structure2_lemma}, the size of $\mathcal{B}$ remains $r$ when restricted to $K$. For any two distinct vertices $u,v \in K$, the edge $uv$ is covered by exactly one biclique $B_j \in \mathcal{B}$. In the addressing $f$ induced by $\mathcal{B}$, this implies $f(u)$ and $f(v)$ differ in exactly one coordinate, specifically at position $j$, where one has $0$ and the other has $1$. Therefore, $d(f(u), f(v)) = 1$ for all distinct $u,v \in K$, making $F_K$ $1$-neighborly.

By Lemma \ref{addressing_lemma_1}, every $u \in K$ has at least one coordinate of $f(u)$ equal to $0$. Thus, the string $1^r$ is not contained in $\bigcup_{u \in K} H(f(u))$, as each subcube $H(f(u))$ requires at least one coordinate fixed to $0$. Since $\{0,1\}^r$ has $2^r$ vertices and $1^r$ is uncovered:
\[
\operatorname{vol}(F_K) = \left| \bigcup_{u \in K} H(f(u)) \right| \leq 2^r - 1 < 2^r.
\]
\end{proof}

\paragraph{}
In Lemma~\ref{balanced_lemma_2}, we provide an explicit construction that yields an upper bound on the minimum size of a biclique partition of a balanced split graph.

    \begin{lemma} \label{balanced_lemma_2}
        The minimum size of a biclique partition of the edge set of the balanced split graph $G$ has size at most $\omega(G)$.
    \end{lemma}
\begin{proof}
A biclique partition of size $\omega(G)$ is constructed by successively removing each vertex of $K$ and its incident edges, and forming the star biclique centered at that vertex at each step. Repeating this process for every vertex in $K$ yields exactly $\omega(G)$ star bicliques. This forms a biclique partition of $G$ of size $\omega(G)$.
\end{proof}

\paragraph{}
In the next lemma, we address an open problem posed by Siewert~\cite{siewert2000biclique}.
\begin{theorem}\label{siewert_open_problem}
There exists a singular tournament on \(n\) vertices with \(r_{\mathbb{Z}^+}(A)=n\), where \(A\) is the adjacency matrix of the tournament.
\end{theorem}
\begin{proof}
We exhibit such a tournament for \(n=9\).
Let \(T\) be the tournament on vertex set \(\{1,\dots,9\}\) with adjacency matrix \(A=(a_{ij})\) given by
\[
A=
\left(
\begin{array}{ccc|ccc|ccc}
0&1&0&0&1&0&0&1&1\\
0&0&1&0&0&1&1&0&1\\
1&0&0&1&0&0&1&1&0\\ \hline
1&1&0&0&1&0&1&1&1\\
0&1&1&0&0&1&1&1&1\\
1&0&1&1&0&0&1&1&1\\ \hline
1&0&0&0&0&0&0&1&0\\
0&1&0&0&0&0&0&0&1\\
0&0&1&0&0&0&1&0&0
\end{array}
\right)
\cong
\left(
\begin{array}{ccc}
P & P & \bar{I}\\
\bar{P} & P & P\\
I & 0 & P
\end{array}
\right).
\]
\noindent
Here \(P\) denotes a (fixed) permutation matrix, \(I\) denotes the identity matrix, and \(J\) denotes the all-ones matrix (so that \(\bar{I}=J-I\) and \(\bar{P}=J-P\)).
\paragraph{}
Let
\(
x=(1,1,1,1,1,1,-1,-1,-1)^{\mathsf T}.
\)
A direct check using the above matrix shows that \(Ax=\mathbf{0}\); hence \(A\) is singular. Furthermore, \(r_{\mathbb{Z}^+}(A)=9\); this was verified computationally using a Gurobi ILP solver. The source code is available in our \href{https://github.com/anandbabunb1/ILP_Binary_rank}{\faGithub \: GitHub repository}.

\end{proof}

\bibliographystyle{plain}
\bibliography{reference_arxiv}

\end{document}